\documentclass{amsproc}

\usepackage{amssymb} 
\usepackage{amsmath}
\usepackage{mathrsfs}

\newtheorem{theorem}{Theorem}[section]
\newtheorem{lemma}[theorem]{Lemma}

\theoremstyle{definition}

\theoremstyle{remark}

\numberwithin{equation}{section}

\newcommand{\R}{{\mathbb R}}

\newcommand{\overbar}[1]{\mkern 1.5mu\overline{\mkern-1.5mu#1\mkern-1.5mu}\mkern 1.5mu}

\author{Mark Allen}
\address{Department of Mathematics, Brigham Young University, Provo,
  UT 84602}
\email{allen@mathematics.byu.edu}


\title[Uniqueness for weak solutions]{Uniqueness for weak solutions of parabolic equations with a fractional time derivative}

\begin{document}

\maketitle
\makeatletter
\vspace{-2em}
{\centering\enddoc@text}
\let\enddoc@text\empty 
\makeatother

\begin{abstract}
 We prove uniqueness for weak solutions to abstract parabolic equations with the fractional Marchaud or Caputo time derivative. We consider weak solutions
 in time for divergence form equations when the fractional derivative is transferred to the test function. 
\end{abstract}

\section{Introduction}
 We prove a uniqueness result for weak-in-time solutions to abstract evolutionary equations. In order to state
 the equation, we recall notation
  from \cite{z09}. Let $\mathscr{V}$ and $\mathscr{H}$ be real separable Hilbert spaces such that $\mathscr{V}$ is densely
 and continuously embedded into $\mathscr{H}$. If $(\cdot, \cdot)$ is the pairing for $\mathscr{H}$ and $\langle \cdot, \cdot \rangle$ 
 is the pairing for $\mathscr{V}' \times \mathscr{V}$ with $\mathscr{V}'$ the dual of $\mathscr{V}$, then we consider the equation 
  \begin{equation}   \label{e:main}
   \begin{aligned}
   &( \partial_t^{\alpha} u(t), v(t)) + a(t,u(t),v(t)) = \langle f(t) , v(t) \rangle,  \\
   &\quad \text{ if } v \in \mathscr{V} \text{  for a.e. } t \in (0,T),
   \end{aligned}
  \end{equation}
   where $\partial_t^{\alpha}$ is the Caputo or Marchaud derivative, 
   and $a:(0,T) \times \mathscr{V} \times \mathscr{V} \to \mathbb{R}$ is a coercive bilinear form.    
    Zacher \cite{z09} proved existence and uniqueness for an equation similar to  \eqref{e:main} but with a different weak formulation than what we will consider and 
    with an even more general fractional derivative.  We will consider a different notion of weak solution made precise in \eqref{e:weak} that naturally accounts for initial data 
    defined on $(-\infty,0]$ as $v(t)$ rather than only defining initial data at the initial time $0$ as $u(0)=v$. 
        
   The evolutionary equation \eqref{e:main} with the fractional time derivative is analogous to abstract evolutionary equations with the usual
   time derivative. One of the most notable examples being the equation 
    \begin{equation}  \label{e:local}
     u_t - (a^{ij}(x,t)u_i)_j = f. 
    \end{equation}
     The function $u$ is a weak solution to the above equation if for every $t \in (0,T)$ and $v \in W_0^{1,2}(\Omega)$, one has
     \begin{equation}  \label{e:strongt}
      \int_{\Omega} u_t v  + a^{ij}(x,t)u_i v_j = \int_{\Omega} fv. 
     \end{equation}
     Solutions to \eqref{e:strongt} are considered weak solutions that are strong in time.   If $u$ is a
     solution to \eqref{e:strongt}, then by integrating in time and transferring the derivative in time to the test function, $u$ is also a solution to 
     \begin{equation}   \label{e:weakt}
      \int_{\Omega} u(x,0)v(x,0)- \int_0^T \int_{\Omega} u v_t + a^{ij}(x,t)u_i v_j = \int_0^T \int_{\Omega} f(x,t)v(x,t)
     \end{equation}
    for all $v \in L^2(0,T;H_0^1(\Omega))$ with $v'  \in L^2(0,T;L^2(\Omega))$ and $v(x,T)\equiv 0$. 
    Proving uniqueness for solutions
    to \eqref{e:strongt} is straightforward since $u$ is a valid test function, so one simply multiplies by $u$ and integrates in time. Proving
    uniqueness for solutions to \eqref{e:weakt} is not straightforward, however, since a priori it is not known if $u$ is a valid test function. In this paper we say ``weak-in-time'' to emphasize that via an integration by parts formula, the fractional time derivative is transferred to the 
    test function as in \eqref{e:weakt}. 
    
    It is advantageous to consider solutions to \eqref{e:weakt}
    which are weak solutions in time since it is often convenient \cite{l67} to work in the larger class of solutions  
    to prove existence as well as regularity estimates such as H\"older continuity. This is also true for parabolic equations involving the Caputo derivative. 
    Zacher \cite{z13} utilized a larger class of solutions to prove H\"older continuity of solutions to parabolic equations involving the Caputo derivative
     which are local in the spatial variables. Recently in \cite{acv16} the author with Caffarelli and Vasseur proved H\"older continuity for solutions to a  nonlocal 
     parabolic equation of divergence form which is a nonlocal analogue to \eqref{e:local} and is given by
      \begin{equation}  \label{e:nonlocal}
       \partial_t^{\alpha} u - \int [w(y,t)-w(x,t)]K(x,y,t)= f(x,t). 
      \end{equation}
     The authors in \cite{acv16} utilized solutions that are weak-in-time to prove H\"older continuity. The same authors also adapted the methods 
     in \cite{acv17} to prove H\"older continuity for a nonlocal porous medium equation. We also remark that H\"older continuity for nondivergence 
     form parabolic equations involving fractional derivatives was shown in \cite{a16} and \cite{a17}. 
          
     Considering weak-in-time solutions is advantageous for existence and regularity results; however,  
     proving uniqueness becomes more difficult. Uniqueness for weak-in-time solutions to \eqref{e:nonlocal} was only shown in 
     \cite{acv16} if $K(x,y,t)$ is independent of $t$. In this paper we show uniqueness for $K(x,y,t)$ with no regularity assumptions on the variable $t$. 
     As the methods apply to the more general parabolic equation \eqref{e:main}, we show uniqueness for weak-in-time solutions to \eqref{e:main}. 
          
    
    In order to formulate weak-in-time solutions, we recall  the definition of the Caputo derivative and an integration by parts formula.

 \subsection{The Marchaud and Caputo Derivatives} \label{s:marchaud}
   The Caputo fractional time derivative for $0<\alpha<1$ is defined as 
    \begin{equation}   \label{e:caputo}
      _a\partial_t^{\alpha} u(t) := \frac{1}{\Gamma(1-\alpha)} \int_a^t \frac{u'(s)}{(t-s)^{\alpha}} \ ds,  
    \end{equation}
   and a reference on the Caputo derivative is \cite{d04}. An equivalent formulation of the Caputo derivative is
   \begin{equation}  \label{e:riemann}
    _a\partial_t^{\alpha} u(t) = \frac{d}{dt} \frac{1}{\Gamma(1-\alpha)} \int_a^t \frac{u(s)-u(a)}{(t-s)^{\alpha}} \ ds. 
   \end{equation}
  This is the formulation of the Caputo derivative utilized in \cite{z09} and \cite{z13}.    
  If $u \in C^1$, then using integration by parts on \eqref{e:caputo} one also has another equivalent formulation  
  \begin{equation}  \label{e:shortM}
    _a\partial_t^{\alpha} u(t) = \frac{1}{\Gamma(1-\alpha)} \frac{u(t)-u(a)}{(t-a)^{\alpha}} 
    + \frac{\alpha}{\Gamma(1-\alpha)} \int_a^t \frac{u(t)-u(s)}{(t-s)^{1+\alpha}} \ ds. 
   \end{equation}
   For simplicity in this paper we multiply $_a\partial_t^{\alpha}$ by the constant $\Gamma(1-\alpha)$. In \eqref{e:main} this 
   constant will be absorbed into $a(t,u(t),v(t))$ and the right hand side $f(t)$. If we define $u(t)=u(a)$ for $t<a$, then up to constant $_a\partial_t^{\alpha}$
    is also equivalent to 
    \begin{equation}  \label{e:marchaud}
    \partial_t^{\alpha} u(t) =  \alpha \int_{-\infty}^t \frac{u(t)-u(s)}{(t-s)^{1+\alpha}} \ ds. 
   \end{equation}
   The formulation in \eqref{e:marchaud} is also known as the Marchaud derivative \cite{skm93} and was recently studied in \cite{bmst16}. We will use the
   Marchaud derivative in \eqref{e:marchaud} in this paper. There are three main advantages to considering the Marchaud derivative. First, it is no longer 
   necessary to consider the initial point $a$ of integration. For this reason we have dropped the subscript $a$ in front of $\partial_t^{\alpha}$. Second, one may 
   consider more general initial data for a parabolic problem. Rather than $u(0)=u_0 \in \mathscr{H}$, one may consider for $t\leq 0$ the 
   initial data $u(t) = v(t)$ where $v(t)(1-t)^{(-1-\alpha)/2} \in L^2(-\infty,0; \mathscr{H})$. Third, the formulation in \ref{e:marchaud} allows one to more easily utilize the nonlocal nature of 
   the fractional time derivative to prove regularity results. This was
   accomplished for divergence form problems in \cite{acv16} and \cite{acv17} and nondivergence problems in \cite{a16} and \cite{a17}. 
   
   For divergence form problems,  the nonlocal nature is utilized by considering an integration by parts formula which also gives a weak-in-time
   formulation of the problem. If $u, \phi \in C^{1}((-\infty,T])$ and $\phi(t)=0$ for $t\leq -M$ for some $M>0$, then 
    \begin{equation}   \label{e:byparts}
     \begin{aligned}
     &\int_{-\infty}^T \phi \partial_t^{\alpha} u \ dt + \int_{-\infty}^T u \partial_t^{\alpha} \phi \ dt \\
      &\quad=  \int_{-\infty}^T \partial_t^{\alpha} ( u \phi ) \ dt 
       + \alpha \int_{-\infty}^T \int_{-\infty}^t \frac{[u(t)-u(s)][\phi(t)-\phi(s)]}{(t-s)^{1+\alpha}} \ ds \ dt.
       \end{aligned}
    \end{equation}
   The Marchaud derivative \eqref{e:marchaud} looks similar to the one-dimensional fractional Laplacian except the integration occurs from only one side. Because of this
   the Marchaud derivative retains some features of the directional derivative. However, the Marchaud derivative also behaves similarly to the fractional Laplacian. This is 
   illustrated by the fact that \eqref{e:byparts} seems to be a combination of 
    \[
       \int \phi \partial_t u  + \int u\partial_t\phi = \int \partial_t(u \phi )
    \]
   and 
   \[
    \int \phi (-\Delta)^{\sigma} u = c_{\sigma} \int \int \frac{[u(x)-u(y)][\phi(x)-\phi(y)]}{|x-y|^{1+2\sigma}} \ dx \ dy. 
   \]
   The first term in the right hand side of \eqref{e:byparts} can be rewritten as shown in \cite{acv16} to accommodate less regular functions, so that 
   the integration by parts formula becomes 
    \begin{equation}   \label{e:byparts3}
     \begin{aligned}
     &\int_{-\infty}^T \phi \partial_t^{\alpha} u \ dt + \int_{-\infty}^T u \partial_t^{\alpha} \phi \ dt \\
      &\quad=  \int_{-\infty}^T \frac{u(t)\phi(t)}{(T-t)^{\alpha}} \ dt 
       + \alpha \int_{-\infty}^T \int_{-\infty}^t \frac{[u(t)-u(s)][\phi(t)-\phi(s)]}{(t-s)^{1+\alpha}} \ ds \ dt.
       \end{aligned}
    \end{equation}
   If we only integrate from $0$ to $T$, then the integration by parts formula takes the form
      \begin{equation}   \label{e:byparts2}
     \begin{aligned}
     \int_{0}^T \phi \partial_t^{\alpha} u
      &=  \int_{0}^T \frac{u(t)\phi(t)}{(T-t)^{\alpha}} \ dt 
       + \alpha \int_{0}^T \int_{-\infty}^t \frac{[u(t)-u(s)][\phi(t)-\phi(s)]}{(t-s)^{1+\alpha}} \ ds \ dt \\
       &\quad - \int_{0}^T u \partial_t^{\alpha} \phi \ dt 
       + \int_{-\infty}^0 u(t)\phi(t)\left[\frac{1}{(T-t)^{\alpha}}- \frac{1}{(0-t)^{\alpha}} \right] \ dt
       \end{aligned}
    \end{equation}
    
 \subsection{Formulation of the weak solution and Main Result}
   With an integration by parts formula in hand we may give a definition of a weak solution. 
   We first recall the definition of the fractional Sobolev space
    \begin{equation}  \label{e:fracsob}
     H^{\alpha/2}(0,T;\mathscr{H}) := \left\{u \in L^2(0,T;\mathscr{H}): \int_0^T \int_0^T \frac{\|u(t)-u(s)\|_{\mathscr{H}}^2}{|t-s|^{1+\alpha}} < \infty \right\}. 
    \end{equation}
    See \cite{dpv12} for a reference of the fractional Sobolev spaces and their properties. 
    The proper initial condition for $u$ will be that $u(t)=v(t)$ for almost all $t \in (-\infty, 0]$ with $v(t)(1-t)^{(-1-\alpha)/2} \in L^2(-\infty,0;\mathscr{H})$. Notice that this initial condition
   includes the possibility that $v(t)$ is time independent so that $v(t)=v \in \mathscr{H}$ which corresponds to the usual initial condition for the Caputo derivative. 
   We then say that $u$ is a weak solution with initial condition $v$, if $u(t)=v(t)$ for almost every $t \leq 0$ and  
   \begin{equation}   \label{e:weak}
    \begin{aligned}
     &\int_{0}^T \frac{(u(t),\phi(t))_{\mathscr{H}}}{(T-t)^{1+\alpha}} \ dt 
      + \alpha \int_{0}^T \int_{-\infty}^t \frac{(u(t)-u(s),\phi(t)-\phi(s))_{\mathscr{H}}}{(t-s)^{1+\alpha}} \ ds \ dt \\
      &\quad - \int_{0}^T (u(t), \partial_t^{\alpha} \phi(t))_{\mathscr{H}} \ dt 
        + \int_0^T a(t,u(t),\phi(t)) \ dt \\
        &\quad + \int_{-\infty}^0 (v(t),\phi(t))_{\mathscr{H}}\left[\frac{1}{(T-t)^{\alpha}}- \frac{1}{(0-t)^{\alpha}} \right] \ dt \\
      &=  \int_0^T \langle f,\phi \rangle_{\mathscr{V}' \times \mathscr{V}} \ dt 
    \end{aligned}
   \end{equation}
   for all 
   \begin{gather*}
    \phi \in L^2(0,T;\mathscr{V}) \\
   \partial_t^{\alpha} \phi \in L^2(0,T; \mathscr{H}) \\
   \phi(t)(1-t)^{-(1+\alpha)/2} \in L^2(-\infty,0;\mathscr{H}) \\
   \phi(t) \equiv 0 \text{ for } t \leq -M \text{ for some } M >0. 
   \end{gather*}
  Since for any test function $\phi$, we assume 
  $\phi(t) \equiv 0$ for $t \leq -M$ for some $M>0$, and since $\phi \in L^2(-M,T;\mathscr{H})$, from \cite{bmst16}
  the Marchaud derivative of $\phi$ can 
  be defined in the distributional sense, and this is how to understand the second requirement above. For the test function we
  will use later, $\partial_t^{\alpha} \phi$  will be defined classically everywhere and will be continuous.

   Our main theorem is the following uniqueness result. 
    
    \begin{theorem}   \label{t:main}
     Solutions to \eqref{e:weak} are unique; i.e. if $u_1,u_2$ are both solutions to \eqref{e:weak} with prescribed initial data $v(t)$
     and right hand side $f(t)$, then $u_1 - u_2 \equiv 0$. 
    \end{theorem}
   
  %
   
   A future question of interest would be to determine necessary and sufficient conditions on the initial data $v(t)$, the bilinear form $a$, and right 
   hand side $f$, so that for a solution $u$ of \eqref{e:weak}, 
   the fractional derivative $\partial_t^{\alpha} u \in L^2(0,T;\mathscr{H}) $ and consequently $u$ is a strong solution to \eqref{e:weak}. Another question of interest 
   is proving uniqueness for weak-in-time solutions for fractional derivatives defined by a kernel $K(t,s)\approx (t-s)^{-1-\alpha}$ and satisfying 
   $K(t,t-s)=K(t+s,t)$. Such kernels were considered in \cite{acv17} and the H\"older continuity results in \cite{acv16} will also apply to fractional derivatives defined
   by such a kernel. 
   
   As this paper is concerned with uniqueness, we do not prove existence of solutions to \eqref{e:weak}. However, we do 
   mention that for a specific $\mathscr{V}$ and $\mathscr{H}$ and a restricted class of right-hand side functions, existence 
   was proven in \cite{acv16}.

 \subsection{Outline and Notation}
  To prove uniqueness we use Steklov averages. In Section \ref{s:steklov} we present preliminary results on how the Steklov averages behave with the fractional 
  Marchaud derivative. In Section \ref{s:unique} we prove our main result. Rather than show that any solution $u$ to \eqref{e:weak} is contained in successively better
  spaces as is done in the local case (see \cite{l67}), we utilize the nonlocal nature of the Marchaud derivative to prove uniqueness directly. This is the essence of 
  Lemma \ref{l:psidelta}.

  We define the notation that will be consistent throughout the paper. 
   \begin{itemize}
    \item $\partial_t^{\alpha}$ - the Marchaud derivative as defined in \eqref{e:marchaud}. 
    \item $\alpha$ - the order of the Marchaud derivative. 
    \item $t,s$ - will always be variables reserved as time variables. 
    \item $H^{\alpha/2}(0,T,\mathscr{H})$ -  the fractional Sobolev space as defined in \eqref{e:fracsob}
    \item $a(t, \cdot, \cdot)$ a bounded $\mathscr{V}$-coercive bilinear form. 
   \item $\lambda, \Lambda$ - The coercivity constants for the bilinear form $a(t, u(t),v(t))$, so that for almost all $t$
    \begin{gather*}
        |a(t,u(t),v(t))| \leq \Lambda \| u(t) \|_{\mathscr{V}} \cdot \| v(t) \|_{\mathscr{V}} \quad  \text{ and } \\
        \lambda \| u(t) \|_{\mathscr{V}}^2 \leq |a(t, u(t),u(t))|. 
    \end{gather*}
    \item $\psi$ a cut-off function defined in \eqref{e:psi}. 
    \item The pairing $(\cdot, \cdot)$ will refer to $(\cdot, \cdot)_{\mathscr{H}}$. 
    \item The norm $\| \cdot \|$ will refer to $\| \cdot \|_{\mathscr{H}}$ unless otherwise stated. 
   \end{itemize}
   

       \section{Steklov averages}   \label{s:steklov}
          In order to apply the Steklov averages technique, we will utilize the following cut-off function $\psi$ throughout the paper. 
           \begin{equation}   \label{e:psi}
            \begin{aligned}
             (1)&\quad \psi: \R \to \R \text{ and } \psi \in C^{\infty}(\R) \\
             (2)&\quad \psi' \leq 0 \\
             (3)&\quad 0 \leq \psi \leq 1 \\
             (4)&\quad \psi(t)=1 \text{ for } t\leq T-2\epsilon \\
             (5)&\quad  \psi(t)=0 \text{ for } t\geq T-\epsilon.
            \end{aligned}
           \end{equation}
          The quantity $\epsilon>0$ will be made precise later. The following Lemma states what equation $u\psi$ will satisfy. We recall that the pairing
          $(\cdot, \cdot)$ represents $(\cdot,\cdot)_{\mathscr{H}}$. 
          \begin{lemma}   \label{l:cut}
           Let $u$ be a solution to \eqref{e:weak} with right hand side $f \equiv 0$ and initial data $v\equiv 0$. Then $\psi(t) u(t)$ satisfies the equation 
            \begin{equation}  \label{e:transfer}
             \begin{aligned}
              & \int_{-\infty}^T \frac{(\psi(t) u(t), \phi(t))}{(T-t)^{\alpha}} \ dt  
              + \alpha
                 \int_{-\infty}^T \int_{-\infty}^s \frac{(\psi(t) u(t) - \psi(s) u(s), \phi(t)-\phi(s))}{(t-s)^{1+\alpha}} \ ds \ dt  \\
             &\quad  -  \int_{-\infty}^T (\psi(t) u(t), \partial_t^{\alpha} \phi(t)) \ dt  
              + \int_{0}^T a(t, \psi u , \phi) \ dt\\
             &= \alpha \int_{-\infty}^T \int_{-\infty}^t \frac{(\phi(t),u(s)[\psi(t)-\psi(s)])}{(t-s)^{1+\alpha}} \ ds \ dt.         
             \end{aligned}
            \end{equation}
          \end{lemma}
          \begin{proof}
            We first show how to transfer $\psi$ from $\phi$ to $u$. 
            We note that if $g, \phi \in C^1(-\infty,T;\mathscr{H})$ with $g(t)=0$ for $t\leq 0$, and $\phi(t)=0$ for 
            $t\leq -M$ for some $M$, then 
            \[
             \begin{aligned}
             & \int_{-\infty}^T \frac{(\psi(t) g(t), \phi(t))}{(T-t)^{\alpha}} \ dt + 
             \alpha \int_{-\infty}^T \int_{-\infty}^t \frac{(\psi(t) g(t)-\psi(s) g(s), \phi(t)-\phi(s))}{(t-s)^{1+\alpha}} \ ds \ dt  \\
             &\quad  - \int_{-\infty}^T (\psi(t) g(t),  \partial_t^{\alpha} \phi(t) ) \ dt \\
             &=  \int_{-\infty}^T ( \phi, \partial_t^{\alpha} (g\psi))  \ dt \\
             &= \int_{-\infty}^T (\phi \psi,  \partial_t^{\alpha} g ) \ dt
               + \alpha \int_{-\infty}^T \int_{-\infty}^t \frac{(\phi(t),g(s)[\psi(t)-\psi(s)])}{(t-s)^{1+\alpha}} \ ds \ dt \\
               &= \alpha \int_{-\infty}^T \int_{-\infty}^t \frac{(\phi(t),g(s)[\psi(t)-\psi(s)])}{(t-s)^{1+\alpha}} \ ds \ dt \\
                &\quad + \int_{-\infty}^T \frac{(g(t), \psi(t) \phi(t))}{(T-t)^{\alpha}} \ dt 
                 + \alpha \int_{-\infty}^T \int_{-\infty}^t \frac{(g(t)-g(s) , \psi(t) \phi(t)-\psi(s) u(s))}{(t-s)^{1+\alpha}}  \ ds \ dt\\
                &\quad - \int_{-\infty}^T (g ,\partial_t^{\alpha} (\psi \phi)) \ dt. \\
             \end{aligned}
            \]
           We then have that
           \[
            \begin{aligned}
             & \int_{-\infty}^T \frac{(\psi g,  \phi)}{(T-t)^{\alpha}} \ dt + 
             \alpha \int_{-\infty}^T \int_{-\infty}^t \frac{(g(t)\psi(t)-g(s)\psi(s), \phi(t)-\phi(s))}{(t-s)^{1+\alpha}} \ ds \ dt  \\
             &\quad  - \int_{-\infty}^T (\psi g,  \partial_t^{\alpha} \phi ) \ dt 
              - \alpha \int_{-\infty}^T \int_{-\infty}^t \frac{(\phi(t),g(s)[\psi(t)-\psi(s)])}{(t-s)^{1+\alpha}} \ ds \ dt \\
             &= \int_{-\infty}^T \frac{(g, \psi \phi)}{(T-t)^{\alpha}} \ dt 
                 + \alpha \int_{-\infty}^T \int_{-\infty}^t \frac{(g(t)-g(s), \psi(t)\phi(t)-\psi(s)\phi(s))}{(t-s)^{1+\alpha}} \ ds \ dt \\
                &\quad - \int_{-\infty}^T (g, \partial_t^{\alpha} (\psi \phi))  \ dt. 
             \end{aligned}
           \]
           We now use that $C^1(0,T; \mathscr{H})$ is dense in $H^{\alpha/2}(0,T;\mathscr{H})$ 
           and since $u(t)= 0$ for $t\leq 0$, we may 
           by approximation substitute $u$ for $g$. The technical point is to show that 
           \begin{equation}  \label{e:g}
            G(t):= \alpha \int_{-\infty}^t \frac{u(s)[\psi(t)-\psi(s)]}{(t-s)^{1+\alpha}} \ ds 
           \end{equation}
           is in $L^2(0,T;\mathscr{H})$. 
           We have that 
           \[
            \| G(t) \|_{\mathscr{H}} \leq \int_{-\infty}^t \frac{\|u(s)\|_{\mathscr{H}}|\psi(t)-\psi(s)|}{(t-s)^{1+\alpha}} \ ds 
            \leq C \int_{T-2\epsilon}^t \frac{\| u(s) \|_{\mathscr{H}}}{(t-s)^{\alpha}} \ ds.  
           \]
            By using Lemma 4.26 from \cite{r96} and the fact that $[T-2\epsilon,T]$ is a finite interval, we have that
           \[
            \int_{T-2\epsilon}^T \left(\int_{T-2\epsilon}^t \frac{\| u(s) \|_{\mathscr{H}}}{(t-s)^{\alpha}} \ ds \right)^2 < \infty. 
           \] 
           Then $G(t) \in L^2(0,T;\mathscr{H})$, and by approximation we obtain that \eqref{e:transfer} holds for $u$. 
           
           Finally, since $\psi(t)$ is constant for each fixed $t$, we have that $a(t,u,\psi \phi)=a(t,\psi u,\phi)$. By integrating
           in time and applying the conclusion from above, we obtain \eqref{e:transfer}.
        \end{proof}

          We define the Steklov averages 
           \[
            \begin{aligned}
                 \eta_{\overbar{h}} &= \frac{1}{h} \int_{t-h}^t \eta\\
                 \eta_{h} &= \frac{1}{h} \int_{t}^{t+h} \eta .
            \end{aligned}
           \]
         Notice that if $\eta(t)$ and $u_h$ both vanish in $[t-2h,T]$, then 
          \begin{equation}  \label{e:transfer2}
           \int_{-\infty}^T (u(t), \eta_{\overbar{h}}(t)) \ dt 
            = \int_{-\infty}^{T-h} (u_h(t) , \eta(t) ) \ dt = \int_{-\infty}^T (u_h(t), \eta(t)) \ dt. 
          \end{equation}

          \begin{lemma}   \label{l:switch}
           If $\partial_t^{\alpha} \eta \in L^2(-\infty,T; \mathscr{H}) $,  and $\eta(t)=0$ for $t<-M$ for some $M>0$, 
           then $\partial_t^{-1} \partial_t^{\alpha} \eta = \partial_t^{\alpha} \partial_t^{-1} \eta$ in $L^2(-\infty,T;\mathscr{H})$ where 
            \[
             \partial_t^{-1} \eta := \int_{-\infty}^t \eta(\tau) \ d \tau. 
            \]
          \end{lemma}
          
          \begin{proof} 
          We first assume that $\eta \in C^1(-\infty,T; \mathscr{H})$ and use the notion of Caputo derivative given in \eqref{e:caputo}. 
           Then 
            \[
             \partial_t^{\alpha} \partial_t^{-1} \eta = c_{\alpha} \int_{-\infty}^t \frac{\eta(s)}{(t-s)^{\alpha}} \ ds.
            \]
            Now 
            \[
             \begin{aligned}
             \partial_t^{-1} \partial_t^{\alpha} \eta &= c_{\alpha}\int_{-\infty}^t \int_{-\infty}^{\tau} \frac{\eta'(s)}{(\tau - s)^{\alpha}} \ ds \ d\tau \\
               &=c_{\alpha}\int_{-\infty}^{t} \int_{s}^{t} \frac{\eta'(s)}{(\tau - s)^{\alpha}} \ d\tau \ ds \\ 
               &= \frac{-c_{\alpha}}{1-\alpha}\int_{-\infty}^{t} \int_{\tau}^{t} \eta'(s)(t-s)^{1-\alpha} \ ds \\
               &= c_{\alpha}\int_{-\infty}^{t}  \frac{\eta(s)}{(t-s)^{\alpha}} \ ds. 
             \end{aligned}
            \]
            Alternatively, one may use the Laplace transform
             \[
               \mathcal{L}[ \partial_t^{-1} \partial_t^{\alpha} \eta ] 
               = s \mathcal{L}[ \partial_t^{\alpha} \eta ]
               = s^{1}s^{-\alpha}\mathcal{L}[  \eta ]
               = s^{-\alpha} \mathcal{L} [\partial_t^{-1} \eta]
               = \mathcal{L}[ \partial_t^{\alpha} \partial_t^{-1}\eta ]. 
             \]
             We then use that $C^1(-\infty,T; \mathscr{H})$ is dense in $L^2(-\infty,T; \mathscr{H})$. 
          \end{proof}

          \begin{lemma}  \label{e:stek}
           If $\eta \in L^2(-\infty,T;\mathscr{H})$ and $\eta(t)=0$ if $t\leq -M$ for some $M$, then 
            \[
               \partial_t^{\alpha} \eta_{\overbar{h}}(t) = (\partial_t^{\alpha} \eta)_{\overbar{h}}(t)
            \]
          \end{lemma}
          
          \begin{proof}
          We notice that 
           \[
            \eta_{\overbar{h}} = \frac{1}{h} \int_{-\infty}^t \eta \ d \tau - \frac{1}{h} \int_{-\infty}^{t-h} \eta \ d \tau.
           \]
           From Lemma \ref{l:switch} the anti-derivative commutes with the Marchaud derivative. 
          \end{proof}
          
          \begin{lemma}  \label{l:conve1}
           Let $ f \in L^2(0,T;\mathscr{V})$ (or $f \in L^2(0,T;\mathscr{H})$). Extend $f(x,t)=0$ for $t \notin [0,T]$. 
           Then $f_h \to f$ in $L^2(0,T;\mathscr{V})$ (or  in $L^2(0,T;\mathscr{H})$).
          \end{lemma}
        
          \begin{proof}
           We suppose that $f \in L^2(0,T;\mathscr{V})$, and remark that the proof when $f \in L^2(0,T;\mathscr{H})$ is the same. 
            \begin{equation}   \label{e:dct}
             \begin{aligned}
              (f_h(t),f_h(t))_{\mathscr{V}}&=  \left(\frac{1}{h}\int_{t}^{t+h} f(s) \ ds, \frac{1}{h}\int_{t}^{t+h} f(s) \ ds \right)_{\mathscr{V}}  \\
             &\leq  \frac{1}{h}\int_{t}^{t+h} (f(s),f(s))_{\mathscr{V}} \ ds  \\
             &= ((f(t),f(t))_{\mathscr{V}})_h.
             \end{aligned}
            \end{equation}
          Now since $f(t)=0$ for $t \notin [0,T]$ we have by changing the order of integration that
          \begin{equation}  \label{e:equal}
           \begin{aligned}
             \int_{0}^T \frac{1}{h}\int_{t}^{t+h} (f(s),f(s)) \ ds \ dt  
            &=  \int_{0}^{T+h} \frac{1}{h}\int_{\max\{s-h,0\}}^{\max\{s-h,T\}} (f(s),f(s)) \ dt \ ds  \\
            &=  \int_{0}^{T}  (f(s),f(s)) \ ds  \\
            &\quad + \frac{1}{h}\int_{T}^{T+h} \int_{s-h}^{\max\{s+h,T\}} (f(s),f(s)) \ ds. 
           \end{aligned}
          \end{equation}
         Thus $\displaystyle \| ((f,f))_h\|_{L^1(0,T; \mathscr{V})} \to \| (f,f) \|_{L^1(0,T;\mathscr{V})}$. Also we have that   $((f,f))_h \to (f,f)$ pointwise for almost everywhere
         $t$. Now  
          $(f_h,f_h) \to (f,f)$ pointwise for almost every $t$ and $(f_h,f_h) \leq ((f,f))_h$ from \eqref{e:dct}. Then  
           from the Lebesgue dominated convergence theorem for the Bochner intergral, we conclude that
          $\|f_h \|_{L^2(0,T;\mathscr{V})} \to \| f \|_{L^2(0,T;\mathscr{V})}$. It then follows that  $f_h \to f$ in $L^2(0,T;\mathscr{V})$. 
          \end{proof}
        
         \begin{lemma}   \label{l:conve2}
          Let $a(t,\cdot, \cdot)$ be a bounded $\mathscr{V}$-coercive bilinear form. Let $u \in L^2(0,T;\mathscr{V})$ with $u(x,t)=0$ for $t \notin [0,T]$.
          Then 
           \begin{equation}  \label{e:dom}
            \lim_{h \to 0} \int_0^T \frac{1}{h} \int_s^{s+h} a(t,\psi u(t), \psi u_h(s)) \ dt \ ds = \int_0^T a(t,\psi u(t),\psi u(t)) \ dt.
           \end{equation}
         \end{lemma}
        
         \begin{proof}
          If $\| \cdot\|$ represents the norm in the space $\mathscr{V}$, then  
           \[
            \begin{aligned}
            &\left| \int_0^T \frac{1}{h} \int_s^{s+h} a(t,\psi u(t), \psi u_h(s)) \ dt \ ds \right| \\
            &\quad\leq \int_0^T \frac{1}{h} \int_{s}^{s+h}\Lambda \|\psi u(t) \| \|\psi u_h(s) \| \ dt \ ds  \\
            &\leq \Lambda^2 \int_0^T \frac{1}{h} \int_{s}^{s+h} \| \psi u_h(s) \|^2 \ dt \ ds  
             +  \int_0^T \frac{1}{h} \int_{s}^{s+h} \|\psi u(t) \|^2 \ dt  \ ds. \\
             &= \Lambda^2 \int_0^T  \| \psi u_h(s) \|^2  \ ds + \int_0^T \frac{1}{h} \int_{s}^{s+h} \|\psi u(t) \|^2 \ dt  \ ds. \\
            \end{aligned}
           \]
            From Lemma \ref{l:conve1} we have that 
           \[
            \lim_{h \to 0} \int_0^T \| u_h(s) \|^2 \ ds = \int_{0}^T \| u \|^2 \ ds, 
           \]
           so it will also be true that 
           \[
           \lim_{h \to 0} \int_0^T \| \psi u_h(s) \|^2 \ ds = \int_{0}^T \| \psi u(s) \|^2 \ ds. 
           \]
           Furthermore, it was shown in the proof of Lemma \ref{l:conve1} that 
           \[
            \lim_{h \to 0} \int_0^T \frac{1}{h} \int_{s}^{s+h} \|\psi u(t) \|^2 \ dt  \ ds = \int_0^T \| \psi u(s) \|^2 \ ds. 
           \]
            From the dominated 
           convergence theorem for the Bochner integral, we conclude \eqref{e:dom} is true. 
          \end{proof}
        
         Before ending this section we establish one more identity. 
         We note that if $u,\eta \in C^1(-\infty,T;\mathscr{H})$  and have compact support in $(-\infty,T-2h)$, then from \eqref{e:transfer2} and \eqref{e:stek} we obtain
         \[
          \begin{aligned}
          &\int_{-\infty}^T (u, \partial_t^{\alpha} \eta_{\overbar{h}})+ ((\partial_t^{\alpha}u), \eta_{\overbar{h}}) \ dt \\
          &= \int_{-\infty}^T (u ,  (\partial_t^{\alpha} \eta)_{\overbar{h}}) + ((\partial_t^{\alpha}u)_h ,\eta) \ dt \\
          &= \int_{-\infty}^T (u_h, \partial_t^{\alpha} \eta ) + ((\partial_t^{\alpha}u)_h, \eta )\ dt \\
          \end{aligned}
         \]
         As a consequence, we immediately have the equality 
         \begin{equation}   \label{e:2transf}
          \begin{aligned}
          &\int_{-\infty}^T \frac{(u ,\eta_{\overbar{h}})}{(T-t)^{\alpha}} \ dt 
           + \alpha \int_{-\infty}^T \int_{-\infty}^t \frac{(u(t)-u(s),\eta_{\overbar{h}}(t)-\eta_{\overbar{h}}(s))}{(t-s)^{1+\alpha}} \ ds \ dt \\
           &=  \int_{-\infty}^T \frac{(u_h ,\eta)}{(T-t)^{\alpha}} \ dt 
           + \alpha \int_{-\infty}^T \int_{-\infty}^t \frac{(u_h(t)-u_h(s),\eta(t)-\eta(s))}{(t-s)^{1+\alpha}} \ ds \ dt.
           \end{aligned}
         \end{equation}
         If $u \in H^{\alpha/2}(0,T;\mathscr{H})$ with $u(t)=0$ for $t\leq 0$ and $T-2h\leq t \leq T$, then by approximation \eqref{e:2transf} will also hold.

           \section{Uniqueness}   \label{s:unique}
            In order to prove uniqueness of solutions, we will utilize the nonlocal structure of the Marchaud derivative. This is the content of the next Lemma. 
            We will use a specific cut-off function $\psi_{\delta}$ which is piecewise linear. 
              \begin{equation}  \label{e:psidelta}
              \psi_{\delta}(t):=
               \begin{cases}
                1 &\text{ if   } \  t\leq T- \epsilon -\delta \\
                \delta^{-1}(T-\epsilon-t) &\text{ if   } \ T-\epsilon -\delta< t < T-\epsilon \\
                0 &\text{ if   } \  t\geq T- \epsilon.  
               \end{cases}
             \end{equation}

           \begin{lemma}   \label{l:psidelta}
            Let $u \in L^2(0,T;\mathscr{H})$, and let $u(t)=0$ for $t\leq 0$. 
            Then for almost all $\epsilon \in (0,T)$, if  $\psi_{\delta}$ is defined as in \eqref{e:psidelta}, then
             \begin{equation}  \label{e:deltazero}
               \lim_{\delta \to 0} \int_{-\infty}^T \int_{-\infty}^t \frac{(\psi_{\delta} u(t), u(s)[\psi_{\delta}(t)-\psi_{\delta}(s)])}{(t-s)^{1+\alpha}} \ ds \ dt = 0. 
             \end{equation}
           \end{lemma}
           
           \begin{proof}
             Throughout this proof the norm $\| \cdot \|$ will refer to $\| \cdot \|_{\mathscr{H}}$.        
             We have
             \begin{equation}  \label{e:delta1}
              \begin{aligned}
              &\left| \int_{-\infty}^T \int_{-\infty}^t \frac{(\psi_{\delta} u(t), u(s)[\psi_{\delta}(t)-\psi_{\delta}(s)])}{(t-s)^{1+\alpha}} \ ds \ dt \right| \\
              &\quad \leq \int_{0}^{T} \|\psi_{\delta} u(t) \| \int_{-\infty}^t \frac{\| u(s) \| |\psi_{\delta}(t)- \psi_{\delta}(s)|}{(t-s)^{1+\alpha}} \ ds \ dt \\
              &\quad=  \int_{T- \epsilon - \delta}^{T-\epsilon} \|\psi_{\delta} u(t) \| \int_{-\infty}^t \frac{\| u(s) \| |\psi_{\delta}(t)- \psi_{\delta}(s)|}{(t-s)^{1+\alpha}} \ ds \ dt. 
              \end{aligned}
             \end{equation}
             The last inequality is due to the fact that $\| \psi_{\delta} u(t)\| =0$ for $t> T- \epsilon - \delta $ and $\psi_{\delta}(t)-\psi_{\delta}(s)=0$ for 
             $t\leq T-\epsilon - \delta$. 
             We define for $0\leq t_0 < t \leq T$, the function
              \[
               F(t,t_0) := \int_{t_0}^t \frac{\|u(s) \|}{(t-s)^{\alpha}} \ ds. 
              \]
              Since $\| u(s) \| \in L^2(0,T)$, then also $\| u(s) \| \in L^1(0,T)$. From Lemma 4.1 in \cite{r96}, for almost every $t \in (0,T)$
             we have $F(t,t_0)\leq F(t,0)<\infty$. Also, as explained earlier in the proof of Lemma \ref{l:cut}, $F(t,t_0) \in L^2(0,T)$ for any $t_0 \in (T-2\epsilon,T)$.  
             We choose $\epsilon$ so that $F(T-\epsilon,0)<\infty$ and $T- \epsilon$ is a Lebesgue point for $\| u \|$ and $F(T-\epsilon, 0)$
             and hence also for $F(T-\epsilon,t_0)$. 
             
              Now if we define
              \[
               H(t):= \int_{-\infty}^t \frac{\| u(s)\| |\psi_{\delta}(t)-\psi_{\delta}(s)|}{(t-s)^{1+\alpha}} \ ds,
              \] 
              then $H(t)=0$ for $t\leq T- \epsilon - \delta$. Furthermore, if $t_0 < T- \epsilon - \delta$, then  for $t \in [T- \epsilon - \delta, T- \epsilon]$ we have
              \begin{equation}  \label{e:delta2}
               \begin{aligned}
               H(t) &\leq 
                 + \frac{1}{\delta} \int_{t_0}^t \frac{\| u(s)\|}{(t-s)^{\alpha}} \ ds \\
                       &= \delta^{-1}F(t,t_0).
                 \end{aligned}
              \end{equation}
              We now multiply by $\| \psi_{\delta}u(t)\|$ and integrate over the variable $t$. 
              Since $T-\epsilon$ is a Lebesgue point for $\| u\|$, we have
              \[
               \lim_{\delta \to 0} \frac{1}{\delta} \int_{T- \epsilon -\delta}^{T-\epsilon} \| \psi_{\delta } u(t) \| F(t,t_0) \leq \| u(T-\epsilon) \| F(T-\epsilon,t_0). 
               \]
               Since $F(T-\epsilon,0)< \infty$ it follows that 
             \begin{equation}  \label{e:delta4}
              \lim_{t_0 \to T-\epsilon} F(T-\epsilon,t_0) \to 0. 
             \end{equation}
             Combining \eqref{e:delta1}, \eqref{e:delta2}, and \eqref{e:delta4} we conclude \eqref{e:deltazero}. 
       \end{proof}

           We now give the proof of the Main Theorem. 
           \begin{proof}[Proof of Theorem \ref{t:main}]       
             Let $u_1,u_2$ be solutions to \eqref{e:weak} with some fixed right hand side $f \in \mathscr{V}'$. Assume also that $u_1(t)=u_2(t)$ for 
             $t \leq 0$. Then $u=u_1(t)-u_2(t)$ is a solution to \eqref{e:weak} with zero right hand side. Furthermore, $u(t)=0$ for $t\leq 0$. 
             
             From Lemma \ref{l:cut} we have that $\psi u$ satisfies \eqref{e:transfer}. We choose $h<\epsilon/2$ 
             so that $\psi u(t)=0$ for $t\geq T-2h$. 
             If we choose $ \eta_{\overbar{h}}$ as a test function with 
             $\eta \in C_0^1(0,T-2h; \mathscr{V})$, then from \eqref{e:2transf}, Lemma \ref{e:stek}, and \eqref{e:transfer2} we obtain  
             \begin{equation}   \label{e:sol1}
             \begin{aligned}
               &\int_{-\infty}^T \frac{((\psi u)_h, \eta)}{(T-t)^{\alpha}} \ dt  
                + \alpha \int_{-\infty}^T \int_{-\infty}^t \frac{((\psi u)_h(t) -(\psi u)_h(s), \eta(t)-\eta(s))}{(t-s)^{1+\alpha}} \ ds \ dt \\
               &\quad - \int_{-\infty}^T (\psi u)_h \partial_t^{\alpha} \eta \ dt  
                 + \int_0^T a(t, \psi u(t), \eta_{\overbar{h}}(t)) \ dt \\
               &= \alpha \int_{-\infty}^T \int_{-\infty}^t \frac{(\eta_{\overbar{h}},u(s)[\psi(t)-\psi(s)])}{(t-s)^{1+\alpha}} \ ds \ dt. 
             \end{aligned}
             \end{equation}
           If $G(t)$ is as defined as in \eqref{e:g}, and $\eta(t)=0$ for $t>T-2h$, then 
            \[
             \int_{-\infty}^T( \eta_{\overbar{h}},G(t)) \ ds = \int_{-\infty}^T (\eta, G_h(t)) \ ds. 
            \]
            We also have that  
            \[
             \begin{aligned}
              \int_0^T a(t, \psi u(t), \eta_{\overbar{h}}(t)) \ dt  
                 &= \int_0^T \frac{1}{h} \int_{t-h}^t a(t,\psi u(t), \eta(s)) \ ds \ dt \\
                 &= \int_0^T \frac{1}{h} \int_s^{\min\{s+h,T\}} a(t,\psi u(t), \eta(s)) \ dt \ ds \\
                 &=\int_0^T \frac{1}{h} \int_s^{s+h} a(t,\psi u(t), \eta(s)) \ dt \ ds. 
              \end{aligned}
             \]
             In the second equality above we have used that $\psi u =0$ for $t \geq T-\epsilon$ and extended $\psi u=0$ for $t \geq T$. 
                Now since $\psi u_h$ is Lipschitz in time,  $\partial_t^{\alpha} (\psi u_h)$ as given in \eqref{e:marchaud} 
            is well-defined and in $L^2(0,T;\mathscr{H})$. Then $(\psi u)_h$ is a valid test function, and if we let  $\eta = (\psi u)_h$ in \eqref{e:sol1}, then we obtain  
             \begin{equation}   \label{e:sol2}
             \begin{aligned}
               &\frac{1}{2}\int_{-\infty}^T \frac{\|(\psi u)_h\|^2}{(T-t)^{\alpha}} \ dt  
                + \frac{\alpha}{2} \int_{-\infty}^T \int_{-\infty}^t \frac{\| (\psi u)_h(t) -(\psi u)_h(s)\|^2}{(t-s)^{1+\alpha}} \ ds \ dt \\
                &\quad + \int_0^T \frac{1}{h} \int_s^{\min\{s+h,T\}} a(t,\psi u(t), (\psi u)_h(s)) \ dt \ ds \\
               &= \alpha \int_{-\infty}^T ((\psi u)_h, G_h(t,s)) \ dt. 
               \end{aligned}
             \end{equation}
            By omitting the first two positive terms we obtain the following inequality
            \[
             \int_0^T \frac{1}{h} \int_s^{\min\{s+h,T\}} a(t,\psi u(t), (\psi u)_h(s)) \ dt \ ds \leq 
                \alpha \int_{-\infty}^T ((\psi u)_h, G_h(t,s)) \ dt.
            \]
            We showed earlier in the proof of Lemma \ref{l:cut} that $ G \in L^2(0,T;\mathscr{H}) $. 
            Since also $\psi u \in L^2(0,T;\mathscr{H})$, we let $h \to 0$ and use Lemmas \ref{l:conve1} and \ref{l:conve2} to conclude 
            \[
             \int_0^T a(t,\psi u, \psi u) \ dt \leq \alpha \int_{-\infty}^T \int_{-\infty}^t \frac{(\psi u(t), u(s)[\psi(t)-\psi(s)])}{(t-s)^{1+\alpha}} \ ds \ dt.
            \]
            If we let $\psi = \psi_{\delta}$ as defined in \eqref{e:psidelta} and let $\delta \to 0$ we obtain from Lemma \ref{l:psidelta} that
            \[
             \int_{0}^{T-\epsilon} \lambda \|u(t)\|_{\mathscr{V}}^2 \ dt  \leq  \int_{0}^{T-\epsilon}a(t, u,u) \ dt \leq 0,
            \]
            where $\lambda$ is the coercivity constant for the bilinear form $a$. 
            Then $u(t)=0$ for $t \leq T-\epsilon$. Since $\epsilon$ can be chosen arbitrarily small, we conclude $u(t)\equiv 0$ on $[0,T]$. 
         \end{proof}

\bibliographystyle{amsplain}
\bibliography{refunique1}

\end{document}